\documentclass[11pt,twoside,reqno]{amsart}
\usepackage[all]{xy} 
\usepackage{enumitem}  
\usepackage {amssymb,latexsym,amsthm,amsmath}
\usepackage[colorlinks, linkcolor=blue, citecolor=red]{hyperref}
\usepackage{graphicx}
\usepackage{mathtools}
\usepackage{xcolor}
\usepackage[utf8]{inputenc}
\usepackage{tikz,quiver}

\long\def\symbolfootnote[#1]#2{\begingroup%
\def\thefootnote{\fnsymbol{footnote}}\footnote[#1]{#2}\endgroup}

\newcommand{\C}{\mathbb C}

\newcommand{\Char}{\textup{char}}

\makeatletter
\def\imod#1{\allowbreak\mkern10mu({\operator@font mod}\,\,#1)}
\makeatother

\newtheorem{theorem}{Theorem}[section]
\newtheorem{lemma}[theorem]{Lemma}

\newtheorem{proposition}[theorem]{Proposition}
\newtheorem*{theorem*}{Theorem}

\theoremstyle{definition}
\newtheorem{remark}[theorem]{Remark}

\newtheorem{question}[theorem]{Question}
\newtheorem{example}[theorem]{Example}

\numberwithin{equation}{section}

\newcommand{\ignore}[1]{}

\newcommand{\mynote}[1]{}

\begin{document}

\setcounter{section}{0}
\setcounter{tocdepth}{1}
\title[Waring Problem for matrices over finite local rings]{Waring Problem for matrices over finite local rings}
\author{Ram Karan Choudhary}
\email{ramkchoudhary1997@gmail.com}

\author{Harish Kishnani}
\email{harishkishnani11@gmail.com}
\author{Anupam Singh}
\email{anupamk18@gmail.com}
\address{Indian Institute of Science Education and Research Pune, Dr Homi Bhabha Road, Pashan, Pune 411008 India}

\thanks{The first-named author acknowledges the support of IISER Pune for the institute postdoctoral fellowship. The second-named author acknowledges the support of ANRF-NPDF grant PDF/2025/002961 for the postdoctoral fellowship. The third-named author is funded by an ANRF-MATRICS Grant ANRF/ARGM/2025/000095/MTR}
\date{\today}

\begin{abstract}
This paper addresses the matrix Waring problem for matrices over finite principal local rings. Let $\mathcal{O}_{\ell}$ be a finite principal local ring of length $\ell$ with the maximal ideal $\mathfrak{m}$ and the residue field $\mathbb{F}_q = \mathcal{O}_\ell/\mathfrak{m}$. When $-1$ is a $k$-th power in $\mathbb{F}_q$ and the characteristic of $\mathbb{F}_q$ does not divide $k$, we show that for sufficiently large $q$, any matrix in $M_n(\mathcal{O}_\ell)$ can be expressed as a sum of two $k$-th powers. Furthermore, we establish that these two conditions are strictly necessary for the result to hold in general.
\end{abstract}
\maketitle

\section{Introduction}
The classical Waring’s problem, originally proposed by Edward Waring, is one of the foundational pillars of number theory. In its classical formulation, the problem asks whether, for every positive integer $k$, there exists a minimal number $g(k)$ such that every natural number can be expressed as a sum of at most $g(k)$ many $k$-th powers of non-negative integers. While Hilbert initially showed that $g(k)$ exists for arbitrary fields, later advancements led to the discovery of explicit formulas for $g(k)$ for all but a finite number of values of $k$. For a detailed survey, see \cite{VW02}. This led to the study of similar questions in non-commutative algebraic structures, such as (a) word maps on algebraic groups, finite simple groups, Lie groups, (b) polynomial maps on central simple algebras, the L'vov-Kaplansky conjecture, the Matrix Waring problem, and (c) Lie polynomial maps on Lie algebras, etc. 

For example, for any odd positive integer $k$, \cite{LOST-Burnside-Theorem} established that every element of a non-abelian finite simple group is expressible as a product of three $k$-th powers. Shalev further generalized this in~\cite{Shalev09} by showing that in sufficiently large-sized finite simple groups, every element is expressible as a product of $3$ evaluations of any given non-identity word. It was later optimized to $2$ evaluations by Larsen, Shalev, and Tiep in \cite{LST11}. In a broader context, the study of word maps on groups has become a highly active area of research within group theory, with particular focus on analyzing their images and the sizes of their fibers (for example, see \cite{LS09, PP15, GKP18-2, GKP18, LST19, EPS24, KKK25}). 

The framework of Waring’s problem also extends naturally to rings, with matrix rings serving as a prime example. The matrix Waring problem asks if, for a given natural number $k$, does there exist $g(k)$, smallest if it exists, such that every element of $M_n(\mathcal R)$ is a sum of $g(k)$ many $k$-th powers where $\mathcal R$ is a commutative ring with unity? In 1985, Newman \cite{Newman85} showed that every element of the matrix ring $M_2(\mathbb Z)$ is a sum of $3$ squares. It was generalised to $M_n(\mathbb Z)$ by Vaserstein \cite{Vaserstein86} in 1986. For a commutative associative ring $R$ with unity, Richman \cite{Richman87} proved that an element of $M_n(R)$ is a sum of $k$-th powers if and only if it is a sum of seven $k$-th powers. Katre and Khule \cite{KaKh00} proved that if $R$ is an order in an algebraic number field $K$, then every element of $M_n(R)$ is a sum of at most seven $k$-th powers if and only if $(k, disc(R)) = 1$. In \cite{BS23}, Bre\v{s}ar and \v{S}emrl demonstrated that if a non-commutative polynomial $f$ is not a polynomial identity or a central polynomial for $M_n(\C)$, then its difference set $f(M_n(\C)) - f(M_n(\C))$ is large enough that any trace-zero matrix can be represented as the sum of two of its members (also see \cite{B04, B20, BS23-2}). Waring's problem can also be generalized to polynomial maps on algebras. A central focus here is the L’vov-Kaplansky conjecture (see \cite{Kaplansky57}), a problem that remains a vibrant topic of research for matrix algebra specialists despite considerable progress (for more details, see \cite{KMR12, BW13, DK16, KMR16, KMRY20}). An analogous problem, motivated by the Waring's problem, has been studied over matrix algebras and Lie algebras (for example, see \cite{DKR21, KS24, KMR25, PSS25, KK25}).

Let $\mathbb {F}_q$ be the finite field of size $q$. In \cite{Kishore22}, Kishore proved that for each integer $k\geq 1$, there exists a constant $C_k$, depending only on $k$ such that for all $q > C_k$; if $n=1,2$, then every element of $M_n(\mathbb{F}_q)$ is a sum of two $k$-th powers, and if $n\geq 3$, then every element of $M_n(\mathbb{F}_q)$ is a sum of at most three $k$-th powers. It was later optimized by Kishore and Singh in \cite{KiSi25}, where they proved that for each integer $k\geq 1$, there exists a constant $C_k$, depending only on $k$ such that for all $q > C_k$ and for all $n \geq 1$, every element of $M_n(\mathbb{F}_q)$ is a sum of two $k$-th powers. The idea to attempt writing every element as a sum of two powers was suggested by Larsen. In our present work, we extend this result to finite principal local rings.

Throughout this article, we let $\mathcal O_\ell$ be a finite principal local ring of length $\ell$ (for example, $\mathbb Z/p^\ell\mathbb Z, \mathbb F_q[x]/\langle x^\ell \rangle$ etc), $\mathfrak m$ be the maximal ideal in $\mathcal O_\ell$ with uniformizer $\pi$ and the residue field $\mathcal O_\ell/\mathfrak m$ be $\mathbb F_q$. We denote the natural map $ \mathcal O_\ell \rightarrow \mathcal O_\ell/\mathfrak m$ by $\theta$. Consider the matrix ring $M_n(\mathcal{O}_\ell)$ with entries in the ring $\mathcal{O}_\ell$. The map $\theta$ induces a map at the matrix level as well, which we simply denote by $\theta$ by abuse of notation. The setup for our problem is as follows. Let $w(x_1, \ldots, x_m) = x_1^k + \cdots + x_m^k$ be a polynomial which gives rise to a polynomial map $w\colon M_n(R)^m \rightarrow M_n(R)$ given by evaluation in both cases when $R=\mathcal O_\ell$ and $R=\mathbb F_q$. Thus, we have the following commutative diagram:
\[\begin{tikzcd}
	{M_n(\mathcal O_\ell)} & {M_n(\mathcal O_\ell)} \\
	{M_n(\mathbb F_q)} & {M_n(\mathbb F_q)}
	\arrow["w", from=1-1, to=1-2]
	\arrow["\theta"', from=1-1, to=2-1]
	\arrow["\theta", from=1-2, to=2-2]
	\arrow["w"', from=2-1, to=2-2].
\end{tikzcd}\]
Note that both vertical maps are surjective, being induced by the natural quotient map. We consider the following question:
\begin{question}
If the bottom map $w$ at $\mathbb F_q$ level is surjective, then is the top map $w$ at $\mathcal O_\ell$ level surjective?
\end{question} 

In particular, due to Kishore and Singh \cite{KiSi25}, we know that for large enough $q$, $w=x_1^k + x_2^k$ is surjective on $M_n(\mathbb F_q)$. So, our question would be to show that if the base field $\mathbb{F}_q$ is large enough, then $w = x_1^k +x_2^k$ is surjective over $M_n(\mathcal{O}_\ell)$? In this article, we answer this question for $q > C_{k, n}$, where $C_{k,n}$ is a constant depending only on $k$ and $n$ (see Remark \ref{remark: C_{k,n}}). Our main theorem is the following:

\begin{theorem}\label{thm:WaringType M_n(O_l)}
Let $\mathcal O_{\ell}$ be a finite principal local ring of length $\ell$. Let $\mathfrak m$ be the maximal ideal with uniformizer $\pi$ and the residue field $\mathcal O_\ell/\mathfrak m \cong \mathbb F_q$. Assume that $-1$ is a $k$-th power of an element in $\mathbb F_q$ and $\Char(\mathbb F_q) \nmid k$. Then, there exists a constant $C_{k, n}$ (depending only on $n$ and $k$), such that for all $q > C_{k,n}$ every matrix of  $M_n(\mathcal O_\ell)$ is a sum of two $k$-th powers in $M_n(\mathcal O_\ell)$.
\end{theorem}

We further demonstrate in Proposition \ref{prop:counter char(F_q)} and Proposition \ref{prop:counter kth root} that the conditions $\Char(\mathbb F_q) \nmid k$ and $-1$ is a $k$-th power in $\mathbb F_q$ are essential and cannot be omitted in general. The paper is organized as follows: Section \ref{sec Waring-type result O-l} extends a result of Small \cite{Small77} to $\mathcal{O}_\ell$ and proves the main theorem for $n=1$. Section \ref{sec: Surjectivity for Mn(O-l)} contains the full proof of our main result. Lastly, Section \ref{sec assumptions are essential} verifies the necessity of the conditions in Theorem \ref{thm:WaringType M_n(O_l)}, showing that the requirements $\text{char}(\mathbb{F}_q) \nmid k$ and $-1$ being a $k$-th power cannot be omitted.

\section{Waring-type result for $\mathcal O_\ell$.}\label{sec Waring-type result O-l}

In this section, we prove a Waring-type result for $\mathcal O_\ell$ that extends a result of Small \cite{Small77} for finite fields.

\begin{theorem} \label{thm:WaringTypeO_l}
Let $\mathcal O_\ell$ be a finite principal local ring of length $\ell$. Let $\mathfrak m$ be the maximal ideal with uniformizer $\pi$ and the residue field $\mathcal O_\ell/\mathfrak m$ be $\mathbb F_q$. Assume that $q$ is sufficiently large (for instance, $q > k^4$) so that every element of $\mathbb F_q$ is a sum of two $k$-th powers in $\mathbb F_q$, for a fixed $k \in \mathbb{N}$. If $-1$ is a $k$-th power of an element in $\mathbb F_q$ and $\Char(\mathbb F_q) \nmid k$, then every element of $\mathcal O_\ell$ is a sum of two $k$-th powers in $\mathcal O_\ell$ such that at least one of them is a unit in $\mathcal O_\ell$.
\end{theorem}
\begin{proof}
We prove it by induction on $\ell$. For $\ell =1$, $\mathcal O_\ell = \mathbb F_q$. If $u \in \mathbb F_q^{\times}$, then by a result of Small \cite{Small77}, it is a sum of two $k$-th powers such that at least one of them is a unit. Further, since $-1$ is a $k$-th power, $0$ is also a sum of two non-zero $k$-th powers. Hence, the theorem holds for $\ell=1$. We now assume that it holds for all finite principal local rings of length less than $\ell$ and proceed to prove it for $\mathcal{O}_\ell$. Fix a section of the natural projection $\mathcal O_\ell \twoheadrightarrow \mathbb O_{\ell-1}$, and identify $\mathbb O_{\ell-1}$ with its image in $\mathcal O_{\ell}$. Let $u \in \mathcal{O}_\ell$. We write $u$ as
$$u = a + \pi^{\ell-1} b,$$
where $a , b\in \mathcal{O}_{\ell-1}$.

Let $\bar u$ be the reduction of $u$ modulo $\langle \pi^{\ell-1} \rangle$. Then by the induction hypothesis, we obtain
$$\bar u = a = \bar x_0^{\,k} + \bar y_0^{\,k}, $$

for some $x_0^{\,k}$ and $x_0^{\,k}$ \text{in } $\mathcal{ O}_{\ell-1}$ such that at least one of $x_0$ or $y_0$ is a unit.

To prove the theorem, it suffices to show that there exist $x, y \in \mathcal O_\ell$ such that
\[
u = x^k + y^k \quad \text{in } \mathcal O_\ell.
\]

Assume that
$x = x_0 + \pi^{\ell-1} x_1$ and $y = y_0 + \pi^{\ell-1} y_1$ 
for some $x_0, y_0 \in \mathbb \pi^{\ell-1}$. Then
\[
u = x^k + y^k 
\iff 
a + \pi^{\ell-1} b = (x_0^k + y_0^k) + \pi^{\ell-1} k \bigl(x_0^{k-1} x_1 + y_0^{k-1} y_1\bigr).
\]
We already have $a = x_0^k + y_0^k$. Equating the coefficients of $\pi^{\ell-1}$ on both sides yields
\begin{equation}\label{eq:O_2}
x_0^{k-1} x_1 + y_0^{k-1} y_1 = k^{-1} b .
\end{equation}

Observe that equation \eqref{eq:O_2} is a non-trivial linear equation in the two variables $x_1$ and $y_1$ (since $x_0$ and $y_0$ are fixed) over $\mathcal O_{\ell-1}$. By induction hypothesis, at least one of $x_0$ and $y_0$ is a unit. Thus, \eqref{eq:O_2} admits a solution over $\mathcal O_{\ell-1}$. Consequently, there exist $x = x_0 + \pi^{\ell-1} x_1$ and $y = y_0 + \pi^{\ell-1} y_1$ in $\mathcal O_\ell$ such that
\[
u = x^k + y^k \quad \text{in } \mathcal O_\ell.
\]
This completes the proof of Theorem~\ref{thm:WaringTypeO_l}.
\end{proof}

The following remark shows that the assumptions in Theorem \ref{thm:WaringTypeO_l} cannot be relaxed.

\begin{remark}
(i). Let $\mathcal{O}_2$ be a finite principal local ring of length $2$ over $\mathbb{F}_3$ with maximal ideal $\mathfrak{m} = \langle \pi \rangle$. If $k=2$, every element of $\mathbb{F}_3$ is a sum of two squares, $\Char(\mathbb{F}_3) \nmid k$, but $-1$ is not a square in $\mathbb{F}_3$. Under these conditions, $\pi$ and $-\pi$ cannot be expressed as a sum of two squares in $\mathcal{O}_2$, demonstrating that the assumption that $-1$ is a $k$-th power in $\mathbb{F}_q$ cannot be omitted from Theorem \ref{thm:WaringTypeO_l} in general.

(ii). Let $\mathcal{O}_3$ be a finite principal local ring of length $3$ over $\mathbb{F}_2$ with maximal ideal $\mathfrak{m} = \langle \pi \rangle$. If $k=2$, every element of $\mathbb{F}_2$ is a sum of two squares, $-1$ is a square in $\mathbb{F}_2$, but $\Char(\mathbb{F}_2) \mid k$. Again, under these conditions, the elements $\pi$ and $-\pi$ cannot be expressed as a sum of two squares in $\mathcal{O}_3$, demonstrating that the assumption that $\Char(\mathbb{F}_2) \nmid k$ cannot be omitted from Theorem \ref{thm:WaringTypeO_l} in general.
\end{remark}

\section{Surjectivity for $M_n(\mathcal O_{\ell})$}\label{sec: Surjectivity for Mn(O-l)}

This section is dedicated to the proof of our main theorem. We begin by establishing a crucial preparatory lemma.

\begin{lemma} \label{lemma:Differential Operator Invertible}
Let $\mathbb{F}$ be a field, $n\ge 1$, and $A\in M_n(\mathbb{F})$ an invertible matrix.
Let $k\ge 1$ be an integer such that $\operatorname{char}(\mathbb{F})\nmid k$.
Define the linear map
\[
D_A : M_n(\mathbb{F}) \longrightarrow M_n(\mathbb{F}),\qquad
D_A(U)=\sum_{i=0}^{k-1} A^{\,i} U A^{\,k-1-i}.
\]
If for every pair of eigenvalues $\alpha,\beta$ of $A$ (in an algebraic closure $\overline{\mathbb{F}}$) with $\alpha\neq\beta$ one has $\alpha^{\,k}\neq\beta^{\,k}$, then $D_A$ is invertible.
\end{lemma}

\begin{proof}
We work over the algebraic closure $\overline{\mathbb{F}}$. Since the determinant of $D_A$ is a polynomial in the entries of $A$ with coefficients in $\mathbb{F}$, it suffices to prove that $D_A$ is invertible over $\overline{\mathbb{F}}$. This implies that $\det(D_A)$ is a nonzero element of $\mathbb{F}$, and hence $D_A$ is invertible over $\mathbb{F}$.

Choose a basis of $\overline{\mathbb{F}}^{\,n}$ so that $A$ is upper triangular. Write
\[
A=
\begin{pmatrix}
\lambda_1 & * & \cdots & *\\
0 & \lambda_2 & \cdots & *\\
\vdots & \vdots & \ddots & \vdots\\
0 & 0 & \cdots & \lambda_n
\end{pmatrix},
\]
where $\lambda_1,\dots,\lambda_n$ are the eigenvalues of $A$ in $\overline{\mathbb{F}}$, counted with algebraic multiplicity. Since $A$ is invertible, we have $\lambda_i \neq 0$ for every $i$.

Let $E_{ab}$ denote the standard matrix unit with a $1$ in the $(a,b)$-position and zeros elsewhere. The collection $\{E_{ab} : 1\le a,b\le n\}$ is a basis of $M_n(\overline{\mathbb{F}})$. Order this basis lexicographically by $(a,b)$, i.e.
\[
E_{11},\, E_{12},\, \dots,\, E_{1n},\, E_{21},\, E_{22},\, \dots,\, E_{nn}.
\]

For each $i$, define the left multiplication map
\[
L_{A^i}: M_n(\overline{\mathbb F}) \to M_n(\overline{\mathbb F}) \quad \text{and} \quad L_{A^i}(X)=A^i X,
\]
and the right multiplication map
\[
R_{A^{k-1-i}}: M_n(\overline{\mathbb F}) \to M_n(\overline{\mathbb F})\quad \text{and} \quad R_{A^{k-1-i}}(X)=X A^{k-1-i}.
\]
Then we have
\[
D_A = \sum_{i=0}^{k-1} L_{A^i} \circ R_{A^{k-1-i}}.
\]

We analyze the action of these maps on the basis $\{E_{ab}\}$. Since $A^i$ is upper triangular with diagonal entries $\lambda_a^i$, its $a$-th column has nonzero entries only in rows $1,\dots,a$. Therefore,
\[
L_{A^i}(E_{ab}) = A^i E_{ab}
\]
is a linear combination of the basis vectors $E_{1b}, E_{2b}, \dots, E_{ab}$.
In the lexicographic order, each of these vectors has index $(r,b)$ with $r\le a$, so $(r,b)\le (a,b)$. Hence, the matrix of $L_{A^i}$ with respect to this ordered basis is upper triangular. Moreover, the coefficient of $E_{ab}$ itself is the $(a,a)$-entry of $A^i$, which is $\lambda_a^i$. Thus, the diagonal entry of $L_{A^i}$ corresponding to $E_{ab}$ is $\lambda_a^i$.

Similarly, since $A^{k-1-i}$ is upper triangular with diagonal entries $\lambda_b^{k-1-i}$, its $b$-th row has nonzero entries only in columns $b,\dots,n$. Hence,
\[
R_{A^{k-1-i}}(E_{ab}) = E_{ab} A^{k-1-i}
\]
is a linear combination of the basis vectors $E_{ab}, E_{a(b+1)}, \dots, E_{an}$.
In the lexicographic order, each of these vectors has index $(a,c)$ with $c\ge b$, so $(a,c)\ge (a,b)$. Therefore, the matrix of $R_{A^{k-1-i}}$ is also upper triangular with respect to this ordered basis, and its diagonal entry corresponding to $E_{ab}$ is $\lambda_b^{k-1-i}$.

For each fixed $i$, the composition $L_{A^i}\circ R_{A^{k-1-i}}$ is upper triangular, and its diagonal entry corresponding to $E_{ab}$ is the product of the diagonal entries
\[
\lambda_a^{i} \cdot \lambda_b^{k-1-i}.
\]
Since $D_A$ is the sum of these compositions over $i=0,\dots,k-1$, it is also upper triangular. Its diagonal entry corresponding to $E_{ab}$ is therefore
\[
\sum_{i=0}^{k-1} \lambda_a^{i} \lambda_b^{k-1-i}
= \mu(\lambda_a,\lambda_b),
\]
where we set
\[
\mu(\alpha,\beta):=\sum_{i=0}^{k-1}\alpha^i\beta^{k-1-i}.
\]
Since the matrix of $D_A$ in this ordered basis is upper triangular, its eigenvalues are exactly its diagonal entries (with algebraic multiplicities). Therefore, the eigenvalues of $D_A$ are precisely
\[
\mu(\lambda_i,\lambda_j), \qquad 1 \le i,j \le n,
\]
where repetitions are counted according to multiplicity.

Next, we show that every eigenvalue $\mu(\lambda_i,\lambda_j)$ of $D_A$ is nonzero under the hypothesis that for any distinct eigenvalues $\alpha,\beta$ of $A$, one has $\alpha^k\neq\beta^k$. We consider two cases.

\paragraph{\bf Case 1 ($\alpha=\beta=\lambda$).} In this case, we have
\[
\mu(\lambda,\lambda)
=
\sum_{i=0}^{k-1}\lambda^i\lambda^{k-1-i}
=
\sum_{i=0}^{k-1}\lambda^{k-1}
=
k\lambda^{k-1}.
\]
Since $\operatorname{char}(\mathbb{F})\nmid k$, the element $k$ is nonzero in $\overline{\mathbb{F}}$. Moreover, $\lambda^{k-1}\neq 0$ because $\lambda\neq 0$. Hence, we get
$ \mu(\lambda,\lambda)\neq 0$.\\

\paragraph{\bf Case 2 ($\alpha\neq\beta$).} In this case, the sum is a finite geometric series
\[
\mu(\alpha,\beta)
=
\alpha^{k-1}+\alpha^{k-2}\beta+\cdots+\beta^{k-1}
=
\frac{\alpha^k-\beta^k}{\alpha-\beta}.
\]
By hypothesis, $\alpha^k\neq\beta^k$ whenever $\alpha\neq\beta$. Since the denominator $\alpha-\beta$ is also nonzero, it follows that $\mu(\alpha,\beta)\neq 0$.

Thus, every eigenvalue $\mu(\lambda_i,\lambda_j)$ of $D_A$ is nonzero. Therefore, all eigenvalues of $D_A$ are nonzero, which implies that $D_A$ is invertible over $\overline{\mathbb{F}}$. Since $\det(D_A)\in\mathbb{F}$ and is nonzero in $\overline{\mathbb{F}}$, it is also nonzero in $\mathbb{F}$. Hence, $D_A$ is invertible over the original field $\mathbb{F}$. This completes the proof of Lemma~\ref{lemma:Differential Operator Invertible}.
\end{proof}

Let $N$ denote the number of $\mathbb{F}_q$-rational points $(x,y)$ satisfying $y^d=f(x)$, where $m=\deg(f)$. A special case of Lang--Weil theorem (see \cite{LW54,Weil49}) asserts that if the polynomial $Y^d-f(X)$ is absolutely irreducible and $q>100dm^2$, then
\[
|N-q|\leq 4d^{3/2}m\sqrt{q}
\]
(see page $10$, \cite{Schmidt76}). We now prove Lemma~\ref{lemma:full orbit}, whose proof is based on Weil's estimate for the number of $\mathbb{F}_q$-rational solutions of the polynomial equation $Y^d-f(X)$.
\begin{lemma} \label{lemma:full orbit}
Let $\mathbb F_q$ be a finite field, let $k\ge 1$ be an integer such that $\operatorname{char}(\mathbb F_q)\nmid k$, and suppose that $-1$ is a $k$-th power in $\mathbb F_q$. Let $d\ge 1$, and let $\alpha\in\mathbb F_{q^d}$. Define
\[
G:=\{t^k : t\in \mathbb F_{q^d}^\times\} \quad \text{and} \quad
S:=\{u\in G : \alpha-u\in G\}.
\]
Assume that
\[
q > \max\left(100k^3,\; \left(2k^{5/2}+\sqrt{4k^5+k^2+2k}\right)^2\right).
\]
Then there exists $u\in S$ such that the Frobenius conjugates
\[
u,\ u^q,\ u^{q^2},\ \dots,\ u^{q^{d-1}}
\]
are pairwise distinct.
\end{lemma}

\begin{proof}
Let $N(\alpha)$ be the number of solutions $(x,y)\in\mathbb F_{q^d}^2$ to
\[
x^k+y^k=\alpha.
\]
If $\alpha=0$, then since $-1\in G$, we have $-u\in G$ for every $u\in G$, so $S=G$. It remains to note that $G$ contains an element whose Frobenius conjugates are pairwise distinct. For $d=1$, this is trivial because $G\ne\varnothing$. For $d\ge2$, let $D:=\gcd(q^d-1,k)$. Then
\[
|G|=\frac{q^d-1}{D}\ge \frac{q^d-1}{k}.
\]
We claim that $q^d-1>2k\,q^{d/2}$. Let $x:=q^{d/2}$. Since $d\ge2$, we have $x\ge \sqrt q$. The hypothesis $q>100k^3$ gives $\sqrt q>10k^{3/2}$, hence $x>10k^{3/2}$.

For every $k\ge1$, we have $\sqrt{k^2+1}\le \sqrt{2}\,k$, so
\[
k+\sqrt{k^2+1} \le (1+\sqrt{2})k < 10k \le 10k^{3/2}.
\]
Thus, we have $x>10k^{3/2} > k+\sqrt{k^2+1}$. The positive root of $X^2-2kX-1$ is exactly $k+\sqrt{k^2+1}$, so $x^2-2kx-1>0$, which is equivalent to
\[
x^2-1>2kx.
\]
Substituting $x=q^{d/2}$ gives $q^d-1>2kq^{d/2}$. Therefore, we have
\[
|G|>\frac{2k\,q^{d/2}}{k}=2q^{d/2}.
\]
The elements of $G$ lying in proper subfields $\mathbb F_{q^e}$ with $e\mid d$, $e<d$, number at most
\[
\sum_{\substack{e\mid d\\ e<d}}(q^e-1)\le \sum_{i=1}^{\lfloor d/2\rfloor} q^i
\le \frac{q}{q-1}q^{d/2}\le 2q^{d/2},
\]
since $q\ge2$. Thus, some $u\in G$ lies outside all proper subfields, and its Frobenius orbit has size $d$. Hence, the conclusion holds for $\alpha=0$. We now assume $\alpha\ne0$.

The polynomial $X^k+Y^k-\alpha$ is absolutely irreducible because $\operatorname{char}(\mathbb F_q)\nmid k$ and $\alpha\ne0$. Hence, Weil's bound (see page $10$, \cite{Schmidt76}) gives
\begin{equation}\label{eq:Nlower}
|N(\alpha)-q^d|\le 4k^{5/2}q^{d/2}.
\end{equation}

Let
\[
D:=\gcd(q^d-1,k).
\]
Every nonzero element of $\mathbb F_{q^d}$ that is a $k$-th power has exactly $D$ preimages under the map $t\mapsto t^k$. For each $u\in S$, there are $D$ choices for $x$ with $x^k=u$ and $D$ choices for $y$ with $y^k=\alpha-u$, yielding $D^2$ solutions. The exceptional cases $x=0$ or $y=0$ contribute at most $2D$ solutions (since each equation $y^k=\alpha$ or $x^k=\alpha$ has at most $D$ solutions when $\alpha\ne0$). Hence, we get
\begin{equation}\label{eq:Nupper}
N(\alpha)\le D^2|S|+2D.
\end{equation}
Combining \eqref{eq:Nlower} and \eqref{eq:Nupper} gives
\begin{equation}\label{eq:Slower}
|S|\ge \frac{q^d-4k^{5/2}q^{d/2}-2D}{D^2}.
\end{equation}

\paragraph{\bf Case 1 ($d=1$).}
If $d=1$, then $\mathbb F_{q^d}=\mathbb F_q$. The Frobenius conjugates of any $u\in\mathbb F_q$ are just the singleton $\{u\}$, which are trivially pairwise distinct. It remains to show that $S\ne\varnothing$. We have already shown
\[
q-4k^{5/2}\sqrt q-2k>0
\]
(this follows from the hypothesis exactly as in the statement of the lemma). From \eqref{eq:Slower}, using $D\le k$, we get
\[
|S| \ge \frac{q - 4k^{5/2}\sqrt{q} - 2D}{D^2}
\ge \frac{q - 4k^{5/2}\sqrt{q} - 2k}{k^2} > 0.
\]
Since $|S|$ is an integer, $|S|\ge1$. Hence, any $u\in S$ satisfies the conclusion.

\paragraph{\bf Case 2 ($d\ge 2$).}
We claim that
\[
|S|>2q^{d/2}.
\]
Using $D\le k$ in \eqref{eq:Slower}, it suffices to prove
\[
\frac{q^d-4k^{5/2}q^{d/2}-2k}{k^2}>2q^{d/2},
\]
which is equivalent to
\begin{equation}\label{eq:main}
q^{d/2}>4k^{5/2}+2k^2+\frac{2k}{q^{d/2}}.
\end{equation}
Since $d\ge2$, we have $q^{d/2}\ge q$. The hypothesis gives $q>100k^3$. We note that for every $k\ge1$,
\[
100k^3 > 4k^{5/2}+2k^2+2.
\]
Indeed, dividing by $k$ gives $100k^2 > 4k^{3/2}+2k+2$, which is immediate (for $k=1$ it is $100>8$, and the difference is increasing for $k\ge1$). Hence, we have
\[
q>100k^3>4k^{5/2}+2k^2+2.
\]
Since $q>k$, we have $2k/q\le2$, so $2k^2+2k/q\le 2k^2+2$. Therefore
\[
q>4k^{5/2}+2k^2+\frac{2k}{q}.
\]
Because $q^{d/2}\ge q$, we get
\[
q^{d/2}\ge q>4k^{5/2}+2k^2+\frac{2k}{q}
\ge 4k^{5/2}+2k^2+\frac{2k}{q^{d/2}},
\]
which proves \eqref{eq:main}. Consequently, we have $|S|>2q^{d/2}$.

We now call an element $u\in S$ \emph{bad} if it lies in a proper subfield $\mathbb F_{q^e}$ with $e\mid d$ and $e<d$. Let
\[
S_0:=\bigcup_{\substack{e\mid d\\ e<d}} (S\cap \mathbb F_{q^e}).
\]
For each proper divisor $e$, the set $G\cap\mathbb F_{q^e}$ has at most $q^e-1$ elements. Hence, we have
\begin{equation}\label{eq:bad}
|S_0|\le \sum_{\substack{e\mid d\\ e<d}} (q^e-1)
\le \sum_{i=1}^{\lfloor d/2\rfloor} q^i
\le \frac{q}{q-1}q^{d/2}
\le 2q^{d/2},
\end{equation}
since $q\ge2$. From $|S|>2q^{d/2}$ and \eqref{eq:bad} we get $|S|>|S_0|$. Therefore, there exists $u\in S\setminus S_0.$

By construction, $u$ is not contained in any proper subfield $\mathbb F_{q^e}$. Equivalently, its minimal polynomial over $\mathbb F_q$ has degree exactly $d$, so the Frobenius conjugates
\[
u,\ u^q,\ u^{q^2},\ \dots,\ u^{q^{d-1}}
\]
are pairwise distinct. This completes the proof of Lemma~\ref{lemma:full orbit}.
\end{proof}

We now prove Lemma~\ref{lemma:Differential Invertible M_n}, which is a key step toward the main theorem. We begin by recalling a standard embedding that reduces the problem to the study of Jordan blocks. Let $R=\mathbb{F}_{q^n}=\mathbb{F}_q(\alpha)$, where $\alpha$ is a primitive element, and let $\{1,\alpha,\ldots,\alpha^{n-1}\}$ be the natural basis of $R$ over $\mathbb{F}_q$. Multiplication by any $\beta\in R$ defines an $\mathbb{F}_q$-linear endomorphism of $R$, yielding an injective ring homomorphism $\phi:R\hookrightarrow M_n(\mathbb{F}_q)$. If $f(t)$ denotes the minimal polynomial of $\alpha$ over $\mathbb{F}_q$, then $\phi(\alpha)$ is the companion matrix $C_f$. Applying $\phi$ entrywise and using the canonical isomorphism $M_d(M_n(\mathbb{F}_q))\cong M_{nd}(\mathbb{F}_q)$, we obtain an embedding $M_d(\mathbb{F}_{q^n})\hookrightarrow M_{nd}(\mathbb{F}_q)$ that maps $J_{\alpha,d}$ to $J_{f,d}$. We define
$$
C_k=\max\left(100k^3,\;\left(2k^{5/2}+\sqrt{4k^5+k^2+2k}\right)^2\right).
$$

\begin{lemma}\label{lemma:Differential Invertible M_n}
Let $\mathbb{F}_q$ be a finite field. Fix $k \in \mathbb{N}$ such that $\operatorname{char}(\mathbb{F}_q) \nmid k$. Assume that $q > C_k$
and $-1$ is a $k$-th power in $\mathbb{F}_q$. Let $X \in M_n(\mathbb{F}_q)$ with
\[
n < \frac{q - 4k^{5/2}\sqrt q - 2k}{k^2}.
\]
Then there exists an invertible matrix $A \in M_n(\mathbb{F}_q)$ such that $X=A^k+B^k$ for some $B \in M_n(\mathbb F_q)$, and the linear map
\[
D_A: M_n(\mathbb F_q)\to M_n(\mathbb F_q),\qquad 
D_A(U)=\sum_{i=0}^{k-1} A^{\,i} U A^{\,k-1-i}
\]
is invertible.  
\end{lemma}

\begin{proof}
We use the primary decomposition (rational canonical form) of $X$ over $\mathbb F_q$. Every matrix $X\in M_n(\mathbb F_q)$ is similar to a block diagonal matrix
\[
Y = \bigoplus_{i=1}^m J_{f_i,r_i},
\]
where each $f_i\in\mathbb F_q[t]$ is a monic irreducible polynomial of degree $d_i$, and $J_{f_i,r_i}\in M_{r_i d_i}(\mathbb F_q)$ is the block corresponding to the invariant factor $f_i^{r_i}$.

For a single such block $J_{f,r}$, let $\deg f=d$. Let $\alpha\in \mathbb F_{q^d}$ be a root of $f$, and let
\[
\phi:\mathbb F_{q^d}\hookrightarrow M_d(\mathbb F_q)
\]
be the regular representation, extended entrywise to
\[
\phi: M_r(\mathbb F_{q^d})\hookrightarrow M_{rd}(\mathbb F_q).
\]
The block $J_{f,r}$ is exactly $\phi(J_{\alpha,r})$, where $J_{\alpha,r}=\alpha I_r+N_r$ is the standard Jordan block over $\mathbb F_{q^d}$.

We first show that for each such pair $(f,r)$, there exist matrices $\mathcal A_{f,r}$ and $\mathcal B_{f,r}$ in $M_{rd}(\mathbb F_q)$ such that
\[
J_{f,r}=\mathcal A_{f,r}^k+\mathcal B_{f,r}^k,
\]
$\mathcal A_{f,r}$ is invertible, and the eigenvalues of $\mathcal A_{f,r}$ have distinct $k$-th powers within the block. The construction is as follows.

Let $G$ be the set of $k$-th powers in $\mathbb F_{q^d}^\times$, and define
\[
S:=\{u\in G : \alpha-u\in G\}.
\]
By Lemma~\ref{lemma:full orbit}, there exists $u\in S$ such that the Frobenius conjugates
\[
u,\ u^q,\ \dots,\ u^{q^{d-1}}
\]
are pairwise distinct. Choose $\lambda_1\in\mathbb F_{q^d}^\times$ with $\lambda_1^k=u$ and $\lambda_2\in\mathbb F_{q^d}$ with $\lambda_2^k=\alpha-u$. Define
\[
A_r:=\lambda_1\left(I_r+\lambda_1^{-k}N_r\right)^{1/k}
= \lambda_1\sum_{j=0}^{r-1}\binom{1/k}{j}\lambda_1^{-kj}N_r^j
\quad \text{and} \quad
B_r:=\lambda_2 I_r.
\]
Then we have $A_r^k=uI_r+N_r$ and $B_r^k=(\alpha-u)I_r$. Therefore, we get $A_r^k+B_r^k=J_{\alpha,r}$. Set $\mathcal A_{f,r}:=\phi(A_r)$ and $\mathcal B_{f,r}:=\phi(B_r)$. The eigenvalues of $\mathcal A_{f,r}$ are exactly the Frobenius conjugates of $\lambda_1$, and their $k$-th powers are $u,u^q,\dots,u^{q^{d-1}}$, which are distinct. Thus, $D_{\mathcal A_{f,r}}$ is invertible on the subspace corresponding to this block.

We now combine the blocks. Let the irreducible factors be $f_1,\dots,f_m$, with roots $\alpha_i\in\mathbb F_{q^{d_i}}$ and associated sets
\[
S_i:=\{u\in G_i : \alpha_i-u\in G_i\},
\]
where $G_i$ is the set of $k$-th powers in $\mathbb F_{q^{d_i}}^\times$.

We claim that we can choose $u_i\in S_i$ for each $i$ such that the Frobenius orbits
\[
\mathcal T_i:=\{u_i^{q^j}:0\le j<d_i\}
\]
are pairwise disjoint. We proceed by induction. Suppose $u_1,\dots,u_{i-1}$ have been chosen. The union of their Frobenius orbits has at most
\[
\sum_{\ell=1}^{i-1} d_\ell \le n
\]
elements. We now distinguish the following two cases.

\paragraph{\bf Case 1 ($d_i\ge 2$).}
In this case, we have
  \[
  |S_i|>2q^{d_i/2}\ge 2q
  \]
  (see the proof of Lemma~\ref{lemma:full orbit}). From our assumption on $n$, we have
  \[
  n < \frac{q - 4k^{5/2}\sqrt q - 2k}{k^2}
  \le \frac{q}{k^2} \le q < 2q,
  \]
  so $|S_i|>n$. Hence, there exists $u_i\in S_i$ whose orbit avoids the previous union.

\paragraph{\bf Case 2 ($d_i=1$).} In this case, we have $S_i\subset \mathbb F_q$. From the  Weil's bound (see page $10$, \cite{Schmidt76}), we have
  \[
  |S_i| \ge \frac{q-4k^{5/2}\sqrt q - 2k}{k^2}
  \]
  (see the proof of Lemma~\ref{lemma:full orbit}). Let $M:=q-4k^{5/2}\sqrt q - 2k$. The hypothesis on $q$ ensures $M>0$. By our assumption on $n$, we get
  \[
  n < \frac{M}{k^2} \le |S_i|.
  \]
  Thus again we can choose $u_i$ avoiding the previous orbits. Consequently, all Frobenius orbits $\mathcal T_i$ are pairwise disjoint.

For each block $i$, construct $\mathcal A_i:=\phi_i(A_{r_i})$ and $\mathcal B_i:=\phi_i(B_{r_i})$ as above, using the chosen $u_i$. Then we have
\[
J_{f_i,r_i} = \mathcal A_i^k+\mathcal B_i^k.
\]
Set
\[
A:=\bigoplus_{i=1}^m \mathcal A_i \quad \text{and} \quad B:=\bigoplus_{i=1}^m \mathcal B_i.
\]
Then we have $Y=A^k+B^k$. The eigenvalues of $A$ are precisely the union over $i$ of the eigenvalues of $\mathcal A_i$, and their $k$-th powers are the union of the Frobenius orbits of the $u_i$, which are pairwise disjoint. Therefore, for any two distinct eigenvalues $\alpha,\beta$ of $A$, we have $\alpha^k\ne \beta^k$. By Lemma~\ref{lemma:Differential Operator Invertible}, the linear map $D_A$ is invertible.

Finally, since $X$ is similar to $Y$, there exists $P\in\mathrm{GL}_n(\mathbb F_q)$ such that
\[
X=PYP^{-1}=(PAP^{-1})^k+(PBP^{-1})^k.
\]
Conjugating $A$ by $P$ preserves the eigenvalue condition, so $D_{PAP^{-1}}$ is similar to $D_A$ and hence invertible. This completes the proof of Lemma~\ref{lemma:Differential Invertible M_n}.
\end{proof}

The following remark gives explicitly the constant $C_{k,n}$ used in our main theorem.
\begin{remark}\label{remark: C_{k,n}}
Assume that $q>C_k$ and
$$n<\frac{q-4k^{5/2}\sqrt{q}-2k}{k^2}.$$
Since the right-hand side of the above relation is increasing for sufficiently large $q$, there exists a constant $C_{k,n}$, depending only on $k$ and $n$, such that whenever $q>C_{k,n}$, the above inequality holds.
\end{remark}

We are now ready to prove our main theorem.

\begin{proof}[Proof of Theorem~\ref{thm:WaringType M_n(O_l)}]
Note that every matrix in $M_n(\mathcal O_\ell)$ is a sum of two $k$-th powers in $M_n(\mathcal O_\ell)$ for $\ell = 1$ (see \cite[Theorem~1.1]{KiSi25}). We prove the result by induction on $\ell$. Assume that the statement holds for $\ell-1$, where $\ell \ge 2$.

Let $X\in M_n(\mathcal O_{\ell})$, and let $X'$ denote its image in $M_n(\mathcal O_{\ell-1})$ under the natural projection
\[
\mathcal O_{\ell}\to\mathcal O_{\ell-1}=\mathcal O_{\ell}/\langle\pi^{\ell-1}\rangle.
\]
By the induction hypothesis, there exist matrices $A',B'\in M_n(\mathcal O_{\ell-1})$ such that
\[
X' = A'^{\,k}+B'^{\,k}.
\]
 Moreover, we may choose $A'$ so that its reduction modulo $\pi$, denoted by $\bar A_0$, satisfies that $D_{\bar A_0}$ is invertible (see Lemma~\ref{lemma:Differential Invertible M_n}). Since this property depends only on the residue class modulo $\pi$, it is preserved under lifting.

Choose arbitrary lifts $A_0,B_0\in M_n(\mathcal O_{\ell})$ of $A'$ and $B'$, respectively. Then we have
\[
X - (A_0^{k}+B_0^{k}) \in \pi^{\ell-1}M_n(\mathcal O_{\ell}),
\]
and hence
\begin{equation} \label{eq:5.4}
    X = A_0^{k}+B_0^{k} + \pi^{\ell-1}Y
    \qquad\text{for some } Y\in M_n(\mathcal O_{\ell}).
\end{equation}

Now, set
\[
A = A_0 + \pi^{\ell-1}U
\quad \text{and} \quad
B = B_0,
\]
where $U \in M_n(\mathcal O_{\ell})$ is to be determined.
Since $\ell\ge 2$, we have $2(\ell-1)\ge\ell$, and therefore $\langle\pi^{\ell-1}\rangle^{2}=0$ in $\mathcal O_{\ell}$. Expanding the $k$-th power and observing that all terms containing at least two factors of $\pi^{\ell-1}$ vanish, we obtain
\begin{align*}
    A^{k}
        &= (A_0+\pi^{\ell-1}U)^{k} \\
        &= A_0^{k}
           + \pi^{\ell-1}\sum_{i=0}^{k-1}
             A_0^{\,i}UA_0^{\,k-1-i} \\
        &= A_0^{k} + \pi^{\ell-1}D_{A_0}(U),
\end{align*}
where $D_{A_0}(U) = \sum_{i=0}^{k-1} A_0^{\,i}UA_0^{\,k-1-i}$. Consequently, we have
\begin{equation} \label{eq:5.5}
    A^{k}+B^{k}
    = A_0^{k}+B_0^{k}
      + \pi^{\ell-1}D_{A_0}(U).
\end{equation}

Comparing \eqref{eq:5.4} and \eqref{eq:5.5}, it suffices to solve
\begin{equation} \label{eq:5.6}
    D_{A_0}(U)=Y
\end{equation}
in $M_n(\mathcal O_{\ell})$.

To this end, reduce \eqref{eq:5.6} modulo $\pi$. Write $\bar A_0=A_0 \pmod{\pi}$, $\bar U=U \pmod{\pi}$, and $\bar Y=Y \pmod{\pi}$. Therefore, we have
\[
D_{\bar A_0}(\bar U)=\bar Y.
\]
Since $A'$ is invertible, $\bar A_0$ is invertible over $\mathbb F_q$, and by construction $D_{\bar A_0}$ is invertible. Therefore,
\begin{align*}
    D_{A_0}(U)=Y \pmod{\pi}
        &\implies D_{\bar A_0}(\bar U)=\bar Y \\
        &\implies \bar U
            = D_{\bar A_0}^{-1}(\bar Y)
\end{align*}
in $M_n(\mathbb F_q)$.

Finally, choose an arbitray lift $U\in M_n(\mathcal O_\ell)$ of $\bar U$. Then we have
\[
D_{A_0}(U)=Y+\pi Z
\]
for some $Z\in M_n(\mathcal O_\ell)$. Substituting this into \eqref{eq:5.5} yields
\[
A^{k}+B^{k}
    = A_0^{k}+B_0^{k}
      + \pi^{\ell-1}(Y+\pi Z).
\]
Since $\pi^\ell=0$ in $\mathcal O_\ell$, the term $\pi^\ell Z$ vanishes, and hence
\[
A^{k}+B^{k}
    = A_0^{k}+B_0^{k}
      + \pi^{\ell-1}Y
    = X
\]
by \eqref{eq:5.4}. Moreover, $A$ is invertible because $A\equiv A_0\pmod{\pi}$ and $\det(A_0)$ is a unit in $\mathcal O_\ell$ (as $\bar A_0$ is invertible). This completes the inductive step. Therefore, the result holds for all $\ell\in\mathbb N$, completing the proof of Theorem~\ref{thm:WaringType M_n(O_l)}.
\end{proof}

\section{On the necessity of the assumptions in Theorem~\ref{thm:WaringType M_n(O_l)}}\label{sec assumptions are essential}

In this section, we prove results demonstrating that the hypothesis $\Char(\mathbb F_q)\nmid k$ and $-1$ is a $k$-th of an element in $\mathbb F_q$ in Theorem~\ref{thm:WaringType M_n(O_l)} is essential and
cannot be omitted. We start by showing that if $\Char(\mathbb F_q)\nmid k$, then the conclusion of Theorem \ref{thm:WaringType M_n(O_l)} does not hold in general.

\begin{proposition} \label{prop:counter char(F_q)}
Let $k\in \mathbb N$ be fixed, and assume that $\operatorname{char}(\mathbb F_q)\mid k$. Fix a section of the natural projection $\mathcal O_\ell \twoheadrightarrow \mathbb \mathcal{O}_{\ell-1}$,
and identify $\mathcal{O}_{\ell-1}$ with its image in $\mathcal{O}_{\ell}$. Then every matrix
$A\in M_n(\mathcal O_\ell)$ can be written uniquely in the form
\[
A=A_1+\pi^{\ell-1} A_2,
\]
where $A_1,A_2\in M_n(\mathcal{O}_{\ell-1})$. Suppose that $\operatorname{tr}(A_2)\neq 0$. Then $A$ is not a sum of two $k$-th powers in $M_n(\mathcal O_\ell)$.
\end{proposition}

\begin{proof}
If possible, on the contrary, assume that there exist $X, Y \in M_n(\mathcal O_\ell)$ such that
\begin{equation} \label{eq: char 1}
    A=A_1+\pi^{\ell-1} A_2 = X^k+Y^k.
\end{equation}
Write $X=X_0+\pi^{\ell-1} X_1$ and $Y=Y_0+\pi^{\ell-1} Y_1$, where $X_0, X_1, Y_0, Y_1 \in M_n(\mathcal{O}_{\ell-1})$. Then
\begin{equation}  \label{eq: char 2}
    X^k+Y^k= (X_0^k+Y_0^k) + \pi^{\ell-1} \left(\sum_{i=0}^{k-1}
             X_0^{\,i}X_1X_0^{\,k-1-i} + \sum_{i=0}^{k-1}
             Y_0^{\,i}Y_1Y_0^{\,k-1-i} \right).
\end{equation}
Equating the coefficients of $\pi$ in equations \eqref{eq: char 1} and \eqref{eq: char 2}, we obtain
\begin{equation} \label{eq: char 3}
    \sum_{i=0}^{k-1}
             X_0^{\,i}X_1X_0^{\,k-1-i} + \sum_{i=0}^{k-1}
             Y_0^{\,i}Y_1Y_0^{\,k-1-i}=A_2.
\end{equation}
Next, we observe that
\begin{align*}
    \operatorname{tr} \left (\sum_{i=0}^{k-1}
             X_0^{\,i}X_1X_0^{\,k-1-i} \right)
    &= \operatorname{tr} \left (\sum_{i=0}^{k-1}
             X_1X_0^{\,k-1} \right) \\
    &= k\,\operatorname{tr}(X_1X_0^{\,k-1}) \\
    &= 0,
\end{align*}
since $\operatorname{char}(\mathbb F_q)\mid k$. Similarly,
\[
\operatorname{tr} \left (\sum_{i=0}^{k-1}
             Y_0^{\,i}Y_1Y_0^{\,k-1-i} \right)=0.
\]
Therefore, taking traces in equation \eqref{eq: char 3}, we obtain
$\operatorname{tr}(A_2)=0$, contradicting the assumption that
$\operatorname{tr}(A_2)\neq 0$.

Hence, our assumption that there exist $X, Y \in M_n(\mathcal O_\ell)$ satisfying $A=X^k+Y^k$ is false. Consequently, the matrix $A=A_1+\pi A_2 \in M_n(\mathcal O_\ell)$ cannot be expressed as a sum of two $k$-th powers in $M_n(\mathcal O_\ell)$. This completes the proof of Proposition~\ref{prop:counter char(F_q)}.
\end{proof}

We now proceed to show that the assumption $-1$ is a $k$-th power in $\mathbb F_q$ is essential in Theorem~\ref{thm:WaringType M_n(O_l)} and cannot be omitted in general. For that, we first establish the following lemma which gives a characterization of matrices in $M_n(\mathbb{F}_q)$ whose $k$-th power is $-I$.

\begin{lemma}\label{lemma:power}
There exists a matrix $X \in M_n(\mathbb{F}_q)$ satisfying $X^k=-I$ if and only if either $q$ is even or $q$ is odd with $v_2(k)<v_2(q^n-1)$, where $v_2(m)$ denotes the exponent of $2$ in the prime factorisation of $m$.
\end{lemma}

\begin{proof}
If $q$ is even, then $-I=I$. Hence, $X=I$ satisfies $X^k=I=-I$ for every $k$.

Now assume that $q$ is odd and let $N=q^n-1=2^t m$ with $m$ odd. Then we have $t=v_2(N)$.
Suppose that $X\in M_n(\mathbb{F}_q)$ satisfies $X^k=-I$. Let $\lambda$ be an eigenvalue of $X$ in $\overline{\mathbb{F}}_q$. Then $\lambda^k=-1$. Since the characteristic polynomial of $X$ has degree $n$, we have $\lambda\in\mathbb{F}_{q^n}$. Thus, $\lambda$ belongs to the cyclic group $\mathbb{F}_{q^n}^{\times}$ of order $N$ and satisfies $\lambda^k=-1$. Therefore, $-1$ is a $k$-th power in $\mathbb{F}_{q^n}^{\times}$.

Since $-1$ has order $2$, it is a $k$-th power in a cyclic group of order $N$ if and only if $2\mid \frac{N}{\gcd(N,k)}$. Equivalently,
$
v_2(N)>v_2(\gcd(N,k)).
$
Since $v_2(\gcd(N,k))=\min\{t,v_2(k)\}$, this is equivalent to $v_2(k)<t=v_2(q^n-1)$.

Conversely, assume that $v_2(k)<t$. Then $-1$ is a $k$-th power in $\mathbb{F}_{q^n}^{\times}$, so there exists $\lambda\in\mathbb{F}_{q^n}^{\times}$ such that $\lambda^k=-1$. Note that if $\lambda\in\mathbb{F}_q$, then $X=\lambda I$ satisfies $X^k=-I$. Further, suppose that $\lambda\notin\mathbb{F}_q$. Let $d=[\mathbb{F}_q(\lambda):\mathbb{F}_q]$. Since $\mathbb{F}_q(\lambda)$ is a subfield of $\mathbb{F}_{q^n}$, we have $d\mid n$ and $d>1$. The Galois conjugates of $\lambda$ are $\lambda, \lambda^q, \dots, \lambda^{q^{d-1}}$. Let $M$ be the companion matrix of the minimal polynomial of $\lambda$ over $\mathbb{F}_q$. Then $M\in M_d(\mathbb{F}_q)$ has eigenvalues $\lambda, \lambda^q, \dots, \lambda^{q^{d-1}}$. Hence, the eigenvalues of $M^k$ are $\lambda^k=-1$, $(\lambda^q)^k=(\lambda^k)^q=-1$, $\dots$, $(\lambda^{q^{d-1}})^k=(\lambda^k)^{q^{d-1}}=-1$, because $q$ is odd. Since $\mathbb{F}_q$ is perfect, the minimal polynomial of $\lambda$ has distinct roots, so $M$ is diagonalizable over $\mathbb{F}_{q^d}$, and therefore $M^k=-I$. Let $X$ be the block diagonal matrix consisting of $\frac{n}{d}$ copies of $M$. Then $X\in M_n(\mathbb{F}_q)$ and $X^k=-I$.

Thus, there exists $X\in M_n(\mathbb{F}_q)$ satisfying $X^k=-I$ if and only if $q$ is even or $v_2(k)<v_2(q^n-1)$ when $q$ is odd. This completes the proof of Lemma~\ref{lemma:power}.
\end{proof}

The following lemma, which is a consequence of Lemma~\ref{lemma:power}, provides a key ingredient in the proof of Proposition~\ref{prop:counter kth root}.
\begin{lemma}\label{lemma:nilpotent}
Assume that $-1$ is not a $k$-th power in $\mathbb{F}_q$ and that $v_2(k)\ge v_2(q^2-1)$, where $v_2(m)$ denotes the exponent of $2$ in the prime factorization of $m$. Let $X,Y\in M_2(\mathbb{F}_q)$. Then
\[
X^k+Y^k=O
\]
if and only if both $X$ and $Y$ are nilpotent.
\end{lemma}

\begin{proof}
Suppose that $X^k+Y^k=O$. We first show that neither $X$ nor $Y$ is invertible. Since $Y^k=-X^k$, $X$ is invertible if and only if $Y$ is invertible. Let $\lambda \neq 0$ be an eigenvalue of $X$, and let $\mu \neq 0$ be the corresponding eigenvalue of $Y$. Since $Y^k=-X^k$, the eigenvalues satisfy
\[
\mu^k=-\lambda^k \implies \left(\frac{\mu}{\lambda}\right)^k=-1
\]
in $\mathbb{F}_{q^2}$. Put $a=\frac{\mu}{\lambda}\in\mathbb{F}_{q^2}^*$. If $a\in\mathbb{F}_q$, then the scalar matrix $aI$ satisfies $(aI)^k=-I$. If $a\notin\mathbb{F}_q$, its minimal polynomial over $\mathbb{F}_q$ is quadratic, and its companion matrix $Z\in M_2(\mathbb{F}_q)$ has eigenvalues $a$ and $a^q$. Since $a^k=-1$ and $(a^q)^k=(-1)^q=-1$ (as $q$ is odd), the polynomial $t^k+1$ vanishes on both eigenvalues of $Z$. Hence, $Z^k=-I$ by Cayley--Hamilton theorem. Thus, in both the cases, there exists $Z\in M_2(\mathbb{F}_q)$ with $Z^k=-I$, contradicting Lemma~\ref{lemma:power}, because $v_2(k)\ge v_2(q^2-1)$. Therefore, $X$ and $Y$ are not invertible.

Since $X$ and $Y$ are singular $2\times2$ matrices over $\mathbb{F}_q$, all of their eigenvalues belong to $\mathbb{F}_q$. Now, assume that $\lambda$ and $\mu$ are nonzero eigenvalues of $X$ and $Y$, respectively. From the relation $X^k+Y^k=O$, we obtain
\[
\lambda^k+\mu^k=0,
\]
and therefore
\[
\left(\frac{\lambda}{\mu}\right)^k=-1
\]
in $\mathbb{F}_q$. Since $-1$ is not a $k$-th power in $\mathbb{F}_q$, this is impossible. Hence, we get $\lambda=\mu=0$. As $X$ and $Y$ are singular, it follows that all eigenvalues of both matrices are zero. Therefore, $X$ and $Y$ are nilpotent.

The converse is immediate, since the $k$-th power of a nilpotent $2\times2$ matrix is zero for $k\ge 2$. Thus, we have $X^k+Y^k=O$ whenever both $X$ and $Y$ are nilpotent. This completes the proof of Lemma~\ref{lemma:nilpotent}.
\end{proof}

The following proposition establishes that the assumption that $-1$ is a $k$-th in $\mathbb F_q$ is essential in Theorem~\ref{thm:WaringType M_n(O_l)} and cannot be omitted in general.

\begin{proposition} \label{prop:counter kth root}
Let $k$ be a positive integer such that $-1$ is not a $k$-th power in $\mathbb F_q$ and $v_2(k)\ge v_2(q^2-1)$, where $v_2(m)$ denotes the exponent of $2$ in the prime factorization of $m$. Fix a section of the natural projection $\mathcal O_2\twoheadrightarrow \mathbb F_q$ and identify $\mathbb F_q$ with its image in $\mathcal O_2$. Then, for every nonzero matrix $A\in M_2(\mathbb F_q)$, the matrix $\pi A\in M_2(\mathcal O_2)$ cannot be expressed as a sum of two $k$-th powers in $M_2(\mathcal O_2)$.
\end{proposition}
\begin{proof}
On contrary, assume that there exist $X, Y \in M_2(\mathcal O_2)$ such that
\begin{equation} \label{eq: kth root 1}
    \pi A = X^k+Y^k.
\end{equation}
Write $X=X_0+\pi X_1$ and $Y=Y_0+\pi Y_1$, where $X_0, X_1, Y_0, Y_1 \in M_2(\mathbb F_q)$. Then
\begin{equation}  \label{eq: kth root 2}
    X^k+Y^k= (X_0^k+Y_0^k) + \pi \left(\sum_{i=0}^{k-1}
             X_0^{\,i}X_1X_0^{\,k-1-i} + \sum_{i=0}^{k-1}
             Y_0^{\,i}Y_1Y_0^{\,k-1-i} \right).
\end{equation}
Comparing equations \eqref{eq: kth root 1} and \eqref{eq: kth root 2}, we obtain
\[
X_0^k+Y_0^k=O
\]
and
\[
\sum_{i=0}^{k-1}
             X_0^{\,i}X_1X_0^{\,k-1-i} + \sum_{i=0}^{k-1}
             Y_0^{\,i}Y_1Y_0^{\,k-1-i}=A.
\]
By Lemma~\ref{lemma:nilpotent}, the relation $X_0^k+Y_0^k=O$ implies that both $X_0$ and $Y_0$ are nilpotent, since $v_2(k)\ge v_2(q^2-1)$. Moreover, the condition $v_2(k)\ge v_2(q^2-1)$ implies that $k \geq 8$. Therefore, for any $i \in \{0, 1, \dots k-1\}$, $\max (i, k-1-i) \geq 2$. Hence,
\[
\sum_{i=0}^{k-1}
             X_0^{\,i}X_1X_0^{\,k-1-i} + \sum_{i=0}^{k-1}
             Y_0^{\,i}Y_1Y_0^{\,k-1-i}=O,
\]
which contradicts the fact that $A$ is a nonzero matrix. Therefore, our assumption that there exist $X, Y \in M_2(\mathcal O_2)$ satisfying $\pi A = X^k+Y^k$ is false. Consequently, the matrix $\pi A \in M_2(\mathcal O_2)$ cannot be expressed as a sum of two $k$-th powers in $M_2(\mathcal O_2)$. This completes the proof of Proposition~\ref{prop:counter kth root}.     
\end{proof}

For a concrete illustration, we present the following example. It demonstrates that the assumption that $-1$ is a $k$-th power in $\mathbb F_q$, appearing in Theorem~\ref{thm:WaringType M_n(O_l)}, is essential and therefore cannot be omitted in general.

\begin{example}\label{example}
For $k=8$ and $n=2$, we have
\[
C_8=\max\left(100\cdot 8^3,\;
\left(2\cdot 8^{5/2}+\sqrt{4\cdot 8^5+8^2+2\cdot 8}\right)^2\right).
\]
Since
\[
100\cdot 8^3=51200,
\]
and
\[
\left(2\cdot 8^{5/2}+\sqrt{4\cdot 8^5+8^2+2\cdot 8}\right)^2
\approx 524447.988,
\]
it follows that
\[
C_8\approx 524447.988.
\]

Suppose that $\mathbb F_q$ has characteristic $3$. Then $q=3^m$ for some positive integer $m$. Since
\[
3^{11}=177147<524447.988<531441=3^{12},
\]
the smallest power of $3$ exceeding $C_8$ is $3^{12}$. Furthermore, we have
\[
2<
\frac{531441-4\cdot 8^{5/2}\sqrt{531441}-16}{64},
\]
so we may take
\[
C_{8,2}=3^{12}=531441.
\]
Hence, Theorem~\ref{thm:WaringType M_n(O_l)} applies whenever $q>531441$.
Now, let $\mathcal O_2/\mathfrak m \cong \mathbb F_{3^{13}}$, and consider the matrix
\[
X=\begin{pmatrix}
\pi & 0\\
0 & 0
\end{pmatrix}\in M_2(\mathcal O_2).
\]
Note that $-1$ is not an $8$-th power in $\mathbb{F}_{3^{13}}$. Since $3^{26}\equiv 9 \pmod{16}$,
it follows from Proposition~\ref{prop:counter kth root} that there do not exist matrices
$A,B\in M_2(\mathcal O_2)$ such that
\[
X=A^8+B^8.
\]
This provides the desired counterexample.
\end{example}










\bibliographystyle{amsalpha}
\bibliography{Reference}
\end{document}